\documentclass[12pt]{amsart}
\usepackage[final]{epsfig}
\usepackage{graphics}
\usepackage{amsmath}
\usepackage{amsthm}
\usepackage{amsfonts}
\usepackage{latexsym}
\usepackage{amssymb}
\usepackage{graphicx}
\usepackage{url}
\usepackage{epstopdf}
\usepackage{xcolor}
\usepackage{anysize}
\marginsize{2cm}{2cm}{2cm}{2cm}

\usepackage{hyperref}
\usepackage{lineno}

\newtheorem{lemma}{Lemma}[section]
\newtheorem{proposition}[lemma]{Proposition}

\newtheorem{theorem}{Theorem}

\newtheorem{corollary}[lemma]{Corollary}

\theoremstyle{remark}
\newtheorem{remark}[lemma]{Remark}

\graphicspath{{fig/}{}}

\begin{document}
\newcommand{\eps}{{\varepsilon}}
\newcommand{\C}{{\mathbb C}}
\newcommand{\Q}{{\mathbb Q}}
\newcommand{\R}{{\mathbb R}}
\newcommand{\Z}{{\mathbb Z}}
\newcommand{\RP}{{\mathbb {RP}}}
\newcommand{\CP}{{\mathbb {CP}}}
\newcommand{\Tr}{\rm Tr}

\title{Billiards in ellipses revisited}

\author{Arseniy Akopyan}
\address{Arseniy Akopyan, Institute of Science and Technology, Klosterneuburg, 3400, N\"O, Austria}
\email{akopjan@gmail.com}

\author{Richard Schwartz}
\address{Richard Schwartz, 
Department of Mathematics,
Brown University,
Providence, RI 02912, USA}
\email{res@math.brown.edu}

\author{Serge Tabachnikov}
\address{Serge Tabachnikov,
Department of Mathematics,
Penn State University,
University Park, PA 16802, USA}
\email{tabachni@math.psu.edu}

\begin{abstract}
	We prove some recent experimental observations of D. Reznik 
concerning periodic billiard orbits in ellipses. For example, the 
sum of cosines of the angles of a periodic billiard polygon remains 
constant in the one-parameter family of such polygons (that exist due to 
the Poncelet porism). In our proofs, we use geometric and complex analytic methods.
\end{abstract}

\date{}
\maketitle

\section{Introduction} \label{sec:intro}

The billiard in an ellipse is a thoroughly studied completely integrable dynamical system, see, e.g., \cite{tabachnikov2005geometry}.
In particular, a billiard trajectory that is tangent to a confocal ellipse will remain tangent to it after each reflection.
That is, the confocal ellipses are the caustics of this billiard. 

One of the properties of this system, a consequence of its complete integrability, is that a periodic
billiard trajectory tangent to a confocal ellipse includes into a 1-parameter family of periodic
trajectories tangent to the same confocal ellipse and having the same period and the same rotation number.
This is the assertion of the Poncelet porism for confocal ellipses.
The Poncelet porism concerns the same kind of $1$-parameter family of
polygons that are simultaneously inscribed and circumscribed in the
same pair of conics, but in general the conics need not be confocal.

A classic result about a continuous $1$-parameter family of billiard paths 
is that their perimeters remain constant.
(See  \cite{tabachnikov2005geometry}, and also Lemma \ref{perimeter} below.)
Recently Dan Reznik conducted a large series of computer experiments with periodic orbits in
elliptic billiards and discovered numerous new properties of these polygons
that are similar in spirit to the constant-perimeter result.  See
\cite{reznik2019applet, reznik2019playlist, reznik2019media, reznik2019elliptic, reznik2020new, reznik2020loci}.  
In this paper we give proofs which verify some of these observations.
Essentially, we prove $3$ main results.  In the body of the
paper, we will also prove a number of variants and generalizations.

We would also like to mention that another proof of the first two results is
presented in \cite{bialy2020danreznik}; it is based on a non-standard generating
functions for convex billiard discovered by Misha Bialy.
\newline
\newline
\noindent
{\bf First Result:\/}
Let $\alpha_1,...,\alpha_n$ be the angles associated to a periodic billiard path.
Figure \ref{fig:angles} shows the case $n=5$.

\begin{figure}[hbtp]
\centering
\includegraphics{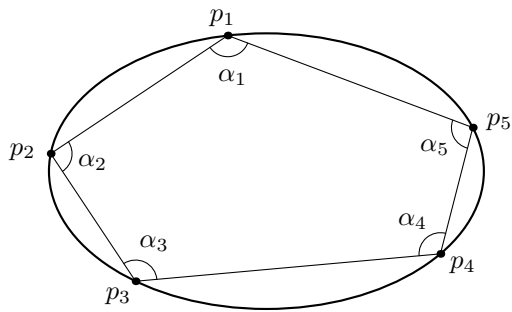} 
\caption{Angles of a billiard trajectory.}
\label{fig:angles}
\end{figure}

\begin{theorem} \label{thm:sumcos2}
The sum
\[
\sum_{i=1}^n \cos\alpha_i
\]
remains constant as $P$ varies a 1-parameter family of periodic billiard paths on an ellipse.
\end{theorem}
This result corresponds exactly to one of Reznik's observations.
We will give two proofs.  The first proof is based on the invariance of the perimeter,
mentioned above, and on the invariance of a quantity called the
{\it Joachimsthal integral\/}.  The second proof is based on Liouville's
Theorem from complex analysis.
\newline
\newline
\noindent
{\bf Second Result:\/}
Let $\beta_1,...,\beta_n$ be the angles of the
polygon formed by tangents to the ellipse at the vertices of a periodic billiard trajectory,
as shown in Figure \ref{fig:prod} for the case $n=5$. 

\begin{figure}[hbtp]
\centering
\includegraphics{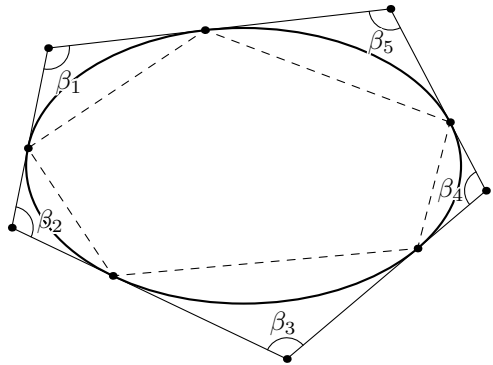}
\caption{The angles involved in Theorem \ref{thm:prodcos}.}
\label{fig:prod}
\end{figure}

\begin{theorem} \label{thm:prodcos}
The the product
\[
\prod_{i=1}^n \cos\beta_i
\]
remains constant as $P$ varies in a 1-parameter family of periodic billiard paths on an ellipse.
\end{theorem}
This result also corresponds exactly to one of Reznik's observations.
Our proof is the same kind of complex analysis argument that we
use for Theorem \ref{thm:sumcos2}.  We don't know a proof along the
lines of our first proof of Theorem \ref{thm:sumcos2} but see \cite{bialy2020danreznik}.
\newline
\newline
\noindent
{\bf Third Result:\/}
Figure~\ref{fig:areas} shows how we construct a polygon $Q$ starting from a Poncelet polygon $P$.

\begin{figure}[hbtp]
\centering
\includegraphics{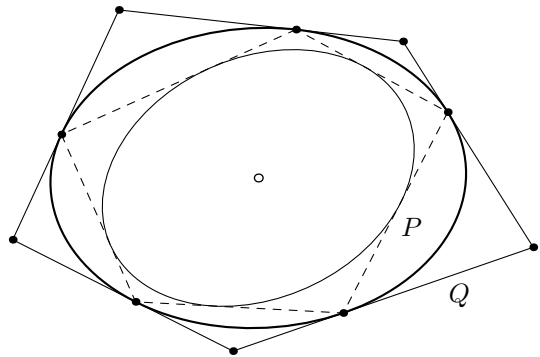}
\caption{To Theorem \ref{thm:areas}.}
\label{fig:areas}
\end{figure}

\begin{theorem} \label{thm:areas}
Let $P$ be a Poncelet $n$-gon with odd $n$, inscribed into an ellipse and circumscribed about
a concentric ellipse.  Let $Q$ be the polygon formed by the tangent lines to the outer ellipse at the vertices of $P$.
The ratio of the areas of the two polygons remains constant within the
Poncelet family containing $P$.
\end{theorem}

This result is a mild generalization of an
observation of Reznik.
We will explain the exact relationship after giving
the proof.  The proof itself is short. It follows
readily from a result called the 
Poncelet Grid Theorem \cite{Schwartz2007poncelet,levi2007poncelet}.

\section{Proof of the First Result} \label{genfacts}

Consider an ellipse in $\R^2$ given by the equation
\begin{equation} \label{eq:ellipse}
\frac{x_1^2}{a_1^2} + \frac{x_2^2}{a_2^2} =1,
\end{equation}
or, in the matrix form, $\langle Ax,x\rangle=1$, where $A = {\rm diag} (1/a_1^2,1/a_2^2)$.  

Recall the notion of polar duality: given a smooth convex closed planar curve $\gamma$, to a
vector $x \in \gamma$ there corresponds a unique covector $x^*$ satisfying the conditions
\[
{\rm Ker}\ x^*=T_x\gamma\ {\rm and}\ (x,x^*)=1;
\]
here the parentheses denote the pairing of vectors and covectors. The points $x^*$ comprise the dual curve $\gamma^*$ in the dual plane.

Identifying vectors and covectors via Euclidean metric, we see that, for the
ellipse (\ref{eq:ellipse}), $x^* = Ax$. The polar dual curve is the ellipse given by the matrix $A^{-1}$.

The phase space of the billiard is $2$-dimensional; it consists of the inward
unit vectors $u$ with the foot point $x$ on the ellipse. The billiard transformation is shown in Figure \ref{fig:transf}. 

\begin{figure}[hbtp]
\centering
\includegraphics{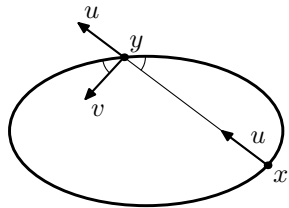} 
\caption{The billiard transformation: $(x,u) \mapsto (y,v)$.}
\label{fig:transf}
\end{figure}

Here is the key fact needed for our first result.  The quantity
$J$ in this result is known as the  Joachimsthal integral.

\begin{proposition} \label{pr:int}
  The function $J(x,u):=-(u, x^*)$ is invariant under the billiard transformation: $J(x,u)=J(y,v)$.
  \end{proposition}

\begin{proof}
	We claim that $(u, x^*)=-(u, y^*) = (v, y^*)$. 
	For the first equality, it is enough to show that $(y-x, x^*)=-(y-x, y^*)$, because $y-x$ is parallel to $u$.
	Since $A$ is self-adjoint operator one has $(y, x^*)=(y,Ax)=(Ay,x)=(x,y^*)$. 
	Subtracting from it the equality $(x, x^*)=(y,y^*)=1$ we obtain the required one.
        Since $u$ is collinear with $y-x$, one has $\langle A(x+y),u\rangle=0$, as needed.
	The second equality holds for billiards of every shape. By the law of billiard reflection,
        the vector $u+v$ is tangent to the ellipse at point $y$, hence  $(y^*, u+v)=0$. This  completes the proof.
\end{proof}

To each billiard trajectory $P$ we can associate the quantity $J(P)=J(x,u)$, where $(x,u)$ is the
first point-vector pair associated to $P$.  

\begin{corollary}
  $J$ is constant as $P$ varies in a $1$-parameter family of periodic billiard paths in an ellipse.
\end{corollary}

\begin{proof}
  Let $E$ be a given ellipse.
  For the generic choice of caustic $E'$ -- i.e., a choice of ellipse confocal with $E$ -- the billiard
map is aperiodic and has dense orbits.  By the preceding result, $J$ is constant
on a dense set of billiard paths tangent to $E'$.   By continuity,
$J$ is constant on the space of all billiard paths tangent to almost any caustic $E'$.
By continuity, the same result holds for every caustic.
Finally, note that the
polygons in a $1$-parameter family of periodic billiard paths are
tangent to the same caustic.
\end{proof}

Here is the perimeter-invariance result mentioned in the introduction.

\begin{lemma}
\label{perimeter}
  The perimiter of $P$ is
  constant as $P$ varies in a $1$-parameter family of periodic billiard paths in an ellipse.
\end{lemma}

\begin{proof}
A $n$-periodic billiard path
extremizes the perimeter among the $n$-gons inscribed in the billiard table.
The 1-parameter family of periodic polygons is a curve in the space of inscribed polygons consisting of
critical points of the perimeter function. A function has a constant value on a curve of its critical points.
See \cite{tabachnikov2005geometry} for more details.
\end{proof}

The following identity, also noticed experimentally by Reznik,
immediately implies Theorem \ref{thm:sumcos2}.
The reason is that both the quantities $J$ and $L$ are
constant for  $1$-parameter families of periodic 
billiard paths in ellipses.

\begin{theorem}
	\label{thm:sumcos1}
For an $n$-periodic billiard trajectory $P$ in an ellipse, one has
\[
\sum_{i=1}^n \cos \alpha_i =  JL -n,
\]
where $L$ and $J$ are the perimeter and the value of the Joachimsthal integral of $P$.
\end{theorem}

\begin{proof}
We compute
\[
L = \sum_i \langle p_{i+1}-p_i,u_i\rangle = \sum_i \langle p_i,u_{i-1}\rangle - \langle p_i,u_i\rangle = \sum_i \langle p_i,u_{i-1}-u_i\rangle 
\]
(this is a discrete version of integration by parts). By the law of billiard reflection,
\[
u_{i-1}-u_i = 2 \sin \left(\frac{\pi-\alpha_i}{2}\right) \frac{p_i^*}{|p_i^*|} = 2 \cos \left(\frac{\alpha_i}{2}\right) \frac{p_i^*}{|p_i^*|}.
\]
Also one has
\[
J = -(u_i,p_i^*) = -|p_i^*| \cos \left(\pi-\frac{\alpha_i}{2}\right) =  |p_i^*| \cos \left(\frac{\alpha_i}{2}\right).
\]
Since $(p_i,p_i^*)=1$, it follows that
\[
JL=  \sum_i 2\cos^2 \left(\frac{\alpha_i}{2}\right) =  \sum_i (1+\cos \alpha_i) = 
 n + \sum_i \cos \alpha_i,
\]
as needed.
\end{proof}

\begin{remark} \label{rmk:Bialy}
See [B] for a symplectic interpretation of these ideas.
\end{remark}

\section{Dual Minkowski billiards} \label{sect:dual}

Before we get to our complex analysis-based proofs, it is useful to discuss
billiards in Minkowski metrics.  This is used to construct dual billiards, which offer new generalizations and help proving the main results.

A~{\it Minkowski metric\/} is a norm on a vector space.
Let $U$ be an $n$-dimensional vector space and $V=U^*$ be its dual.
 Assume that $U$ and $V$ are equipped with Minkowski metrics, not necessarily centrally-symmetric 
and dual to each other. Let $M \subset U$ be the unit co-ball of the metric in 
$V$ and $N \subset V$ be he unit co-ball of the metric in $U$. 
One has two billiards: $M$ in normed space $U$, and $N$ in normed space $V$. 

\begin{lemma}
The two billiards systems are isomorphic.
\end{lemma}

\begin{proof}
This is proved in \cite{gutkin2002billiards}, section 7.  Here we sketch the ideas.
Abusing notation, denote by $\mathcal{D}$ the polar duality that identifies the unit 
spheres and co-spheres of the metrics. The phase space of the billiard in $M$ consists of 
the pairs $(q,u)$ where $q\in \partial M$ and $u$ is a (Minkowski) unit inward vector at $q$. 
Assign to it the pair $(q,p) \in \partial M \times \partial N$ where $p=\mathcal{D}(u)$. 
Likewise, the phase space of the billiard in $N$ consists of the pairs $(p,v)$ where 
$p\in \partial N$ and $v$ is a (Minkowski) unit inward vector at $p$. Assign to it 
the pair $(p,q) \in \partial N \times \partial M$ where $q=\mathcal{D}(v)$.

The assertion  is that a sequence $(\ldots,q_i,p_i,q_{i+1},\ldots)$ corresponds to a 
billiard orbit in $M$ if and only if the sequence $(\ldots,p_i,q_{i+1},p_{i+1},\ldots)$ corresponds to a billiard orbit in $N$.  
If $M$ is an ellipse with half-axes $a_1$ and $a_2$ in the Euclidean plane $U$, then $N$ is 
the unit disc in the plane $V$ whose unit ball is polar to $M$, i.e., is an ellipse with half-axes $1/a_1$ and $1/a_2$.
\end{proof}

\begin{remark}
The affine map $(x_1,x_2)\rightarrow (a_1x_1, a_2x_2)$ isometrically maps the 
Minkowski plane $V$ to the Euclidean plane, taking $N$ back to the ellipse with half-axes $a_1$ and $a_2$, that is, to $M$.
 One obtains a symmetry of the billiard in an ellipse, called the skew hodograph 
transformation and discovered by A. Veselov  \cite{veselov1988integrable,veselov1991integrable}.
\end{remark}

\begin{remark}
It was noticed in \cite{arstein-avidan2014bounds} that minimal action of a billiard in 
Minkowski metric is related with Ekeland--Hofer--Zehnder capacity, which made it possible 
to apply symplectic geometry to convex geometry problems. In particular it was shown that 
the famous Mahler conjecture followed from the Viterbo's 
conjecture~\cite{arstein-avidan2014fromsympelctic, akopyan2016elementary}.
\end{remark}

\section{Another Proof of the First Result}

As before, let $u_1,\ldots,u_n$ be the unit vectors along the sides of $P$. These vectors are points of
the unit circle $\omega$ centered in the origin $O$, and they form an $n$-periodic trajectory 
of the Minkowski billiard in $\omega$ with the Minkowski metric defined by the ellipse dual 
to the original one. Therefore $u_1,\ldots,u_n$ are the vertices of a Poncelet $n$-gon, 
inscribed in $\omega$ and circumscribed about some ellipse $\xi$ centered at $O$.

\begin{figure}[hbtp]
\centering
\includegraphics{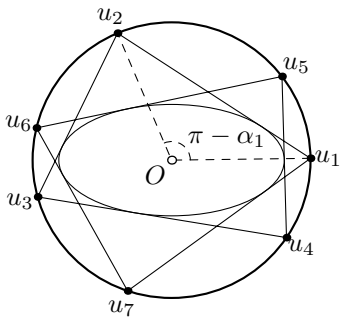}
\caption{To  Theorems ~\ref{thm:circle-ellipse1} and \ref{thm:circle-ellipse2}.}
\label{inter}
\end{figure}

The angles of $P$  satisfy  $\alpha_i = \pi - \angle u_{i-1} O u_{i}$,
thus 
the next result is equivalent to Theorem \ref{thm:sumcos2}.  

\begin{theorem}
	\label{thm:circle-ellipse1}
	Let $u_1,\ldots, u_n$ be a Poncelet polygon inscribed in a circle $\omega$ with center $O$ and circumscribed about an ellipse $\xi$ with the same center. Then 
\[
\sum_{i=1}^n \cos \angle u_i O u_{i+1}
\]	
is constant in the 1-parameter family of Poncelet $n$-gons.
\end{theorem}

\begin{proof}
Following the approach of \cite{schwartz2015pentagram}, we will complexify the situation, that is, 
extend the setting to Poncelet polygons on the conics given by the same equations in the complex plane. 
We show that the function in question is  bounded, and then the Liouville theorem  implies that this function is constant.
To extend our function to complex plane, we need to represent the function $\cos \angle u_i O u_{i+1}$ in a 
more convenient way. Since $|Ou_i|=1$ for all $i$, we have:
\[\cos \angle u_i O u_{i+1}= \langle u_i, u_{i+1} \rangle .\]
In other words, for the proof of the first statement, we need to show that the sum $\langle u_i, u_{i+1} \rangle$ is constant. 
Let us emphasize that here we consider the usual dot product, not the Hermitian one. 

Consider the standard rational parametrization of the circle $\omega$:
\[
	p(t)=\left(\frac{2t}{t^2+1}, \frac{t^2-1}{t^2+1}\right).
\]
The points at infinity correspond to the value of the parameter $t= \pm I$. 
Here $I=\sqrt{-1}$.
It is clear that the only possibility for the Poncelet polygon to have an 
infinite $\sum\langle u_i, u_{i+1} \rangle$ is to have one of its vertex at infinity.

Let us show that when a vertex goes to infinity, the inner 
products in which this vertex participates cancel each other. 

Consider a point at infinity, say,  $p(I)$. We claim that its two neighboring vertices
 of the Poncelet polygon, denoted by $a(I)$ and $b(I)$, are opposite points of $\omega$.
Indeed, the lines $p(I)a(I)$ and $p(I)b(I)$ are tangent to $\xi$ and are parallel, 
therefore the tangency points of these lines with $\xi$ are symmetric with respect to $O$, 
and hence their intersection points with $\omega$ are also symmetric (the point $p(I)$ is invariant under the reflection in $O$, given by $t \mapsto -1/t$). 
Thus, for any finite point $q$ on $\omega$, we have 
$\langle q, a(I) \rangle + \langle q, b(I) \rangle=0.$

Now, consider point $p(t+I)$ with $t$ tending to zero and its neighboring 
vertices $a(I+t)$ and $b(I+t)$ of the Poncelet polygon. Notice that $p(t+I)$ tends to 
infinity as $O(1/t)$, while $a(t+I)$, $b(t+I)$ tend to their limit $a(I)$, $b(I)$ linearly. 
Furthermore, due to the symmetry, 
as $t$ goes to zero, the linear in $t$ terms are vectors with the same absolute value and the opposite directions:
\[
	a(t+I)=a(I)+\vec k  \cdot t +O(t^2), \ \ 
	b(t+I)=-a(I)-\vec k  \cdot t +O(t^2).
\]
Now we can bound above the sum for small $t$:
\begin{multline*}
	\label{eq:sum of neighbors}
	\langle
	p(t+I), a(t+I)
	\rangle
	+
	\langle
	p(t+I), b(t+I)
	\rangle
	=
	\langle
	p(t+I),a(t+I)+b(t+I)
	\rangle
	=\\
	\langle
	O(1/t), O(t^2)
	\rangle= O(t).
\end{multline*}
That is, the sum tends to zero as $t$ goes to $0$, and therefore it is bounded.
\end{proof}

\begin{figure}[hbtp]
\centering
\includegraphics{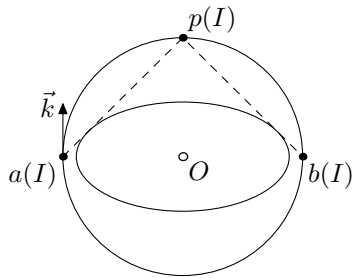}
\caption{The behavior of the polygon at infinity.}
\label{infin}
\end{figure}
\medskip

\section{Proof of the Second Result}

Referring to the construction in the previous section,
the angles $\beta_i$ in Figure \ref{fig:prod} are given by the formula
\[
\beta_i = \pi - \frac{\angle u_{i-1} O u_{i+1}}{2}=\angle u_{i-1}u_{i}u_{i+1}.
\]

\begin{theorem}
	\label{thm:circle-ellipse2}
	Let $u_1,\ldots, u_n$ be a Poncelet polygon inscribed in a circle $\omega$ with center $O$ and circumscribed about an ellipse $\xi$ with the same center. Then
\[
\ \prod_{i=1}^n \cos  \angle u_{i-1}u_{i}u_{i+1}
\]	
is constant in the 1-parameter family of Poncelet $n$-gons.
\end{theorem}

\begin{proof}
Since the product of cosines changes continuously, 
if the absolute value of this product is constant in the family, then its sign is fixed as well.
Therefore, instead of the product of cosines, we may consider the product of their squares:
\[
\cos^2 \angle u_{i-1} u_{i} u_{i+1}= \cos^2\frac{\angle u_{i-1}Ou_{i+1}}{2}=\frac{1+ \cos {\angle u_{i-1}Ou_{i+1}}}{2} = \frac{1+\langle u_{i-1}, u_{i+1} \rangle}{2}.
\]
Thus we need to prove that the product $\prod (1+\langle u_{i-1}, u_{i+1} \rangle)$ is bounded.
Again the only possibility for this product to be infinite is when one of the vertices goes to infinity.

Similarly to the previous proof, and using the same parameterization of the circle,
we assume that $p(t+I)$ is the vertex that goes to infinity, $a(t+I)$ and $b(t+I)$ are its neighboring vertices, and $a'(t+I)$ and $b'(t+I)$ are its second removed neighbors (which are also centrally symmetric, but it plays no role here).

Let us show that the corresponding product is bounded:
\begin{multline*}
	\label{eq:sum of neighbors}
(1+ \langle 	p(t+I), a'(t+I) \rangle)
(1+ \langle 	p(t+I), b'(t+I) \rangle)
(1+ \langle 	a(t+I), b(t+I) \rangle)=\\
O(t^{-1})\cdot O(t^{-1})\cdot (1+\langle a(I)+\vec k  
\cdot t +O(t^2), -a(I)-\vec k  \cdot t +O(t^2) \rangle)=\\
O(t^{-2})\cdot(1+\langle a(I), -a(I)\rangle -2 \langle a(I), \vec k \cdot t \rangle + O(t^2))=\\
O(t^{-2})\cdot(-2 \langle a(I), \vec k \cdot t \rangle + O(t^2))=
O(t^{-2})\cdot(-2 \langle a(I), \vec k \cdot t \rangle )+ O(1).
\end{multline*}
It is left to notice that $\vec k$ is tangent to $\omega$ at $a(I)$, 
therefore $\langle a(I), \vec k \cdot t \rangle =0.$
Thus  the product is bounded.

If $n$ is  odd, then only one vertex can go to infinity: if $p(\pm I)$ is a vertex of the Poncelet polygon, then $p(\mp I)$ is not its vertex. 

For even $n$, a vertex $u_i$ goes to infinity simultaneously with $u_{i+n/2}$.
A simple combinatorial analysis of the configurations shows that the only case when 
the above considered factors  coincide is when $n=4$.  
In that case the polygon is always a rectangle, and the statement is obvious.
\end{proof}

\section{Variants and Generalizations} \label{var1}

Here we list some variants and generalizations of the results we
have proved so far.
\newline
\newline
{\bf Multi-dimensional version:\/}
{\rm The billiard inside an ellipsoid in $\R^{n+1}$ is also completely integrable: the phase space is $2n$-dimensional, 
the trajectories are confined to $n$-dimensional tori, and the motion on these tori is quasi-periodic, 
see \cite{dragovic2011poncelet}. In particular, if a point is periodic, then all points of 
the torus are periodic with the same period, and the respective polygons have the same perimeters. 
The billiard map still has the Joachimsthal integral, constant on the orbits confined to an 
invariant torus, and the above arguments  go through, proving a multi-dimensional version of Theorem \ref{thm:sumcos2}.
\newline
\newline
\noindent
{\bf Sizes of the Angles:\/}
Concerning Theorem \ref{thm:prodcos},
it was pointed out to us by M.~Bialy that this theorem implies 
that the sign of the quantity $\beta_i - \pi/2$ remains fixed during the rotation of the polygon. 
 This  is consistent with the fact that all vertices of $Q$ lie on a ellipse polar dual to 
the inner ellipse (the caustic) with respect to the outer one. If we fix this outer ellipse 
and vary the caustic, then these polar dual ellipses, into which the polygons $Q$ are inscribed, 
form a pencil of conics. This pencil contains the \emph{orthoptic circle}, the locus of points 
from which an ellipse is seen under the right angle. This orthoptic circle separates 
the two cases: when all angles $\beta_i$ are obtuse and when they are all acute. 
\newline
\newline
{\bf Additional Invariants:\/}
The following theorem is a generalization of Theorem \ref{thm:sumcos2}, which is its $k=1$ case. We explain how to deduce this theorem from Theorem \ref{thm:sumcos2}.

\begin{theorem} \label{thm:gensumcos}
For each $k=1,...,n$, the quantity
\[
C_k = \sum_{i=1}^n \cos (\alpha_i + \alpha_{i+1} + \ldots + \alpha_{i+k-1})
\]
remains constant as $P$ varies a 1-parameter family of periodic billiard paths on an ellipse.
\end{theorem}

\begin{figure}[hbtp]
\centering
\includegraphics[height=3in]{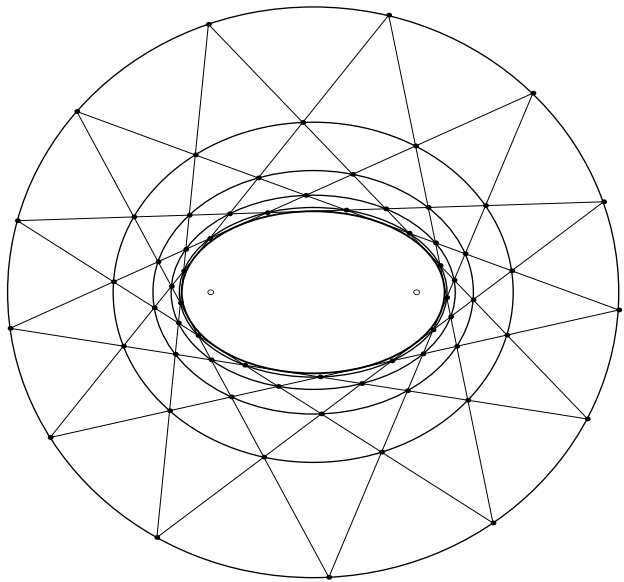} 
\caption{Four $13$-periodic billiard orbits in confocal ellipses tangent to the same caustic.}
\label{pentagons}
\end{figure}

\begin{proof}
To line up our proof with a previously published result that we use, we
state things in terms of Poncelet polygons and the Poncelet porism.
Label the lines containing the sides of a Poncelet $n$-gon cyclically.
Fix $k$ and consider  
the intersections of $i$th and $(i+k)$th lines, where $i=1,2,\ldots,n$. This set of points 
lies on a confocal ellipse and comprises several polygons, each a periodic billiard trajectory 
(the number of polygons equals gcd $(n,k)$).  This statement is a part of the Poncelet 
Grid theorem \cite{Schwartz2007poncelet,levi2007poncelet}. Figure~\ref{pentagons} illustrates the case of $n=13$.
The angles of these new polygons are expressed via the angles of the original one. Namely, the new angles are equal to
\[
\alpha_i + \alpha_{i+1} + \ldots + \alpha_{i+k-1} - (k-1)\pi.
\]
Therefore Theorem \ref{thm:sumcos2}, applied to the new polygons, implies that $C_k$ remains constant in their Poncelet family.
\end{proof}

Theorem \ref{thm:gensumcos} has a reformulation, generalizing Theorem \ref{thm:circle-ellipse1}, which is the $k=1$ case.

\begin{theorem}
	\label{thm:gensumcos2}
	Let $u_1,\ldots, u_n$ be a Poncelet polygon inscribed in a circle $\omega$ with center $O$ and circumscribed about an ellipse $\xi$ with the same center. Then, for each $k$,
\[
\sum_{i=1}^n \cos \angle u_i O u_{i+k}
\]	
are constant in the 1-parameter family of Poncelet $n$-gons.
\end{theorem}

\noindent
{\bf Sums of Squared Lengths:\/}
Here is a  corollary of Theorem \ref{thm:gensumcos2}.
Note that \[|u_{i+k}-u_i|^2= 2-2 \cos \angle u_i O u_{i+k}.\] This implies

\begin{corollary}
	\label{cor:sides}
Consider a Poncelet polygon inscribed in a circle and circumscribed about a concentric ellipse. 
Then the sum of the squared  lengths of its $k$-diagonals remains constant in the 1-parameter family of Poncelet polygons.
\end{corollary}

\noindent
{\bf Product of sines of half-angles:\/}
Notice that the angles of polygons formed by the tangents can be 
represented through angles of the billiard trajectory:
\[
\beta_i=\frac{\alpha_i+\alpha_{i+1}}{2}.
\]
Therefore, the statement of Theorem~\ref{thm:prodcos} can be formulated in terms of
the angles of the billiard trajectory and then can 
be extended for the angles of lines in the corresponded
Poncelet grid as in Theorem~\ref{thm:gensumcos}.
Doing this for odd $n$ and the polygon formed by sides of our trajectory
 with step $k=(n-1)/2$, we find that the corresponding angle of 
the tangential polygon $\beta'_i$ equals 
\[
\beta'_i=\frac{\sum_{j\ne i} \alpha_j}{2}=\frac{\pi(n-2)-\alpha_i}{2}.
\]
Since $n$ is odd we get
\[
\cos \beta'_i=\cos\frac{\pi(n-2)-\alpha_i}{2}= \pm \cos \frac{\pi-\alpha_i}{2} =\pm\sin \frac{\alpha_i}{2}.
\]

This gives us the following result, also noticed by D.~Reznik:
\begin{corollary}
	\label{cor:product of halfsines}
	For odd $n$, the quantity
	\[
	\prod_{i=1}^n \sin \frac{\alpha_i}{2} 
	\]
	remains constant as $P$ varies in a 1-parameter
        family of periodic billiard paths on an ellipse.
\end{corollary}

\noindent
{\bf Hyperbolic Interpretation:\/}
One can interpret Theorems \ref{thm:gensumcos} and \ref{thm:circle-ellipse2} in terms of hyperbolic geometry. 
Consider $\omega$ as the absolute of the Klein model of the hyperbolic plane, and $\xi$ as 
an ellipse in it. The  hyperbolic and the Euclidean measures of the angles $u_i O u_k$
coincide. We obtain the following corollary.

\begin{corollary}
	\label{cor:poncelet}
	Let $u_1,\ldots, u_n$ be an ideal $n$-gon in the hyperbolic plane whose sides are tangent to an ellipse with center $O$. Then, for each $k$,
\[
\sum_{i=1}^n \cos \angle u_i O u_{i+k}\ \ {\rm and} \ \ \prod_{i=1}^n \cos  \left(\frac{\angle u_{i-1} O u_{i+1}}{2}\right)
\]	
are constant in the 1-parameter family of ideal Poncelet $n$-gons.
\end{corollary}

\noindent
{\bf A Dual Version:\/}
Under the  duality transform with respect to the circle $\omega$, the inner ellipse $\xi$ goes to a 
concentric ellipse $\xi^*$, and we again obtain a Poncelet polygon, this time inscribed in $\xi^*$ and 
circumscribed about $\omega$. The angles between the unit vectors $u_i$ become the angles between 
the sides of the Poncelet polygon, and we obtain the following corollary of Theorems \ref{thm:circle-ellipse1}
and \ref{thm:circle-ellipse2}.

\begin{figure}[hbtp]
\centering
\includegraphics{cosines-fig-8.mps}
\caption{To Corollary~\ref{cor:ellipse-circle}. }
\label{exter}
\end{figure}

\begin{corollary}[Ellipse-Circle version]
	\label{cor:ellipse-circle}
	Let $v_1,\ldots, v_n$ be a Poncelet polygon circumscribed about a circle and inscribed in a concentric ellipse. Denote its angles by $\alpha_i$. Then
\[
\sum_{i=1}^n \cos \alpha_i\ \ {\rm and} \ \ \prod_{i=1}^n \cos  \angle v_{i}Ov_{i+1}
\]	
are constant in the 1-parameter family of  Poncelet $n$-gons.
\end{corollary}

\section{Proof of the Third Result}

Theorem \ref{thm:areas} is closely related to the
experimental observation of D. Reznik illustrated in Figure \ref{grid}. 
Consider an $n$-periodic billiard trajectory in an ellipse with odd $n$ (the pentagon $Q$ in Figure~\ref{grid}). 
The tangent lines to the ellipse at these  $n$ points form a new $n$-gon (the pentagon $P$ in Figure~\ref{grid}). 
The observation is that {the ratio of the areas of these two polygons remains 
constant as the $n$-periodic billiard trajectory varies in its 1-parameter family.}

\begin{figure}[hbtp]
\centering
\includegraphics{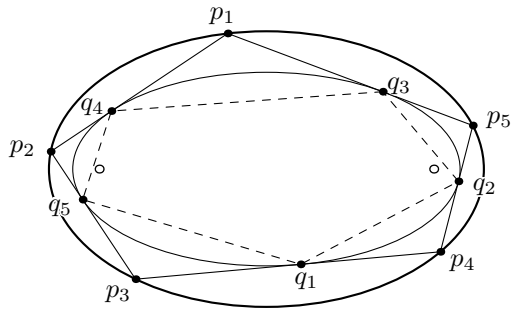}
\caption{The ratio of the areas of the polygons $P$ and $Q$ remains constant in the 1-parameter family of Poncelet polygons.}
\label{grid}
\end{figure}

There exists an affine transformation that makes the two ellipses confocal: first turn one of 
the ellipses into a circle by an affine transformation, and then stretch along the axes of the other ellipse to make the two confocal.
Since an affine transformation does not affect the ratio of the areas, it remains 
to prove the claim for a Poncelet polygon on confocal ellipses, see Figure \ref{grid}.
This is the case observed by Reznik.

Now we deal with the confocal case.
The Poncelet grid theorem \cite{levi2007poncelet} implies that the affine transformation 
that takes the inner ellipse to the outer one by scaling its main axes and reflecting 
in the origin takes the inner polygon to the outer one as well (it maps each vertex 
to the ``opposite one", see the labelling in Figure \ref{grid}). Since the ratio of the areas is invariant under an affine transformation, the result follows.


\bigskip
{\bf Acknowledgements}. This paper would not be written if not for Dan Reznik's curiosity and 
persistence; we are very grateful to him. We also thank R. Garcia and J. Koiller for interesting discussions.
It is a pleasure to thank  the Mathematical Institute 
of the University of Heidelberg for its stimulating atmosphere. ST thanks M.~Bialy 
for interesting discussions and the Tel Aviv University for its invariable hospitality.

AA was supported by European Research Council (ERC) under the European Union's 
Horizon 2020 research and innovation programme (grant agreement No 78818 Alpha). 
RS is supported by NSF Grant DMS-1807320.
ST was supported by NSF grant DMS-1510055 and SFB/TRR 191.

\bibliographystyle{abbrvurl}
\bibliography{cosinelib.bib}
\vskip 2cm

\end{document}